\documentclass[11pt]{article}
 \usepackage{amsfonts}
\usepackage[utf8]{inputenc}
\usepackage{amsmath,amsthm,amsfonts,amsfonts,amssymb,color}
\usepackage{stmaryrd}
\usepackage{mathrsfs}
\usepackage{makeidx}
\usepackage{graphicx}
\graphicspath{{./}}
\allowdisplaybreaks
\usepackage{amsthm}
\usepackage{mathrsfs}
\newtheorem{theorem}{Theorem}[section]

\newtheorem{lemma}{Lemma}[section]

\def\E{\mathbb{E}}
\def\H{\mathcal{H}}
\def\3{\interleave}
\def\G{\Gamma}
\def\R{\mathbb R}
\def\<{\langle} \def\>{\rangle}

\def\kk{\kappa}
\begin{document}

\title{\textbf{Matrix-valued SDEs arising from currency exchange markets}}
\author{Panpan Ren$^2$\footnote{corresponding author} and Jiang-Lun Wu$^{1,2,}$ \\
{\small\it $^1$ School of Mathematics, Northwest University, Xi'an, Shaanxi 710127, P R China}  \\
{\small\it $^2$ Department of Mathematics, Swansea University, Swansea SA2 8PP, UK} \\
{\small Email: 673788@swansea.ac.uk; j.l.wu@swansea.ac.uk}  }

\date{}

\maketitle

\begin{abstract}
In this paper, motivated by modelling currency exchange markets with matrix-valued stochastic processes, matrix-valued stochastic differential equations (SDEs) 
are formulated. This is done based on the matrix trace, as for the purpose of modelling currency exchange markets. To be more precise, we set up a Hilbert 
space structure for $n\times n$ square matrices via the trace of the Hadamard product of two matrices. With the help of this framework, one can then define 
stochastic integral of It\^o type and It\^o SDEs. Two types of sufficient conditions are discussed for the existence and uniqueness of solutions to the 
matrix-valued SDEs.  

 \end{abstract}



\section{Introduction} 
Analysing currency trading, or to be more precise, modelling, predicting and hedging foreign currency exchange rates are important problems, giving that modern communication 
technology nowadays brings unified (global) financial market setting worldwide, making the currency exchange markets more and more complicated on one side and highly demanding 
deep mathematical analysis and high computing in the micro level model on the other side. This then challenges theoretical considerations as well as computing technology profoundly. 
The latter is linked to deep machine learning and data analysis, see, e.g. \cite{RenWu} and references therein. 

The objective of the present paper is to set up a theoretical yet basic framework for the market dynamics of currency exchange markets with stochastic volatility, in which we aim to 
formulate matrix-valued SDEs for the market dynamics structure. Matrix-valued SDEs appear naturally as an important extension of vector-valued SDEs, as discussed by many 
authors, see e.g. \cite{IkedaW,RevuzYor,Oksendal,Protter} and references therein. There are some considerations of matrix-valued SDEs applying to finance, cf. for instance 
\cite{Jaschke, DuanYan,LiqingYan}, just mention a few. The main difficult for matrix-valued SDEs is non-commutativity, as one needs to consider either left or right matrix-valued 
stochastic processes, or even employing Lie group and Lie algebra structures or random matrices such as Dyson Brownian motion, as investigations of matrix-valued SDEs 
in diffusion particle systems e.g. \cite{KatoriTanemura}. Here we formulate another type matrix-valued SDEs, yet simple but useful for currency market modelling, via certain symmetry operations of matrices. It is hoped that our framework could be useful for certain parameter (matrix) estimates with data analysis which we plan to carry out in our forthcoming work. 

Let $\mathbb{R}^n\otimes\mathbb{R}^n$ be the totality of $n\times n$ square (real) matrices. Based on the observation of currency exchange 
markets, we use the matrix trace to define the Hilbert-Schmidt inner product $\langle\cdot,\cdot\rangle_{HS}$ (via the Hadamard product) of any two matrices so that 
$(\mathbb{R}^n\otimes\mathbb{R}^n$,$\langle\cdot,\cdot\rangle_{HS})$ becomes a Hilbert space. This features the starting point of our study in this paper. 
With this Hilbert space in hand, we can define It\^o stochastic integral for (adapted) square matrix-valued stochastic processes against matrix-valued Brownian motions and 
hence we can formulate matrix-valued stochastic integral equations rigorously.  We are mainly concerned with the existence and uniqueness of the solutions of such matrix-valued 
SDEs. We will present two types of sufficient conditions for the existence and uniqueness results, those sufficient conditions are extensions of corresponding conditions 
in \cite{AlbBrzWu,TrumanWu}, respectively.  

The rest of the paper is organised as follows. Preliminaries on currency market features, (square) matrices and the Hilbert space structure, matrix-valued Brownian motion and It\^o 
stochastic calculus are presented in Section 2.  We then establish, in Section 3, the existence and uniqueness of solutions to matrix-valued SDEs under the Lipschitz and growth 
conditions (in terms of $||\cdot||_{HS}$ norm).  Section 4,  the final section, is devoted to the derivation of It\^o formula for the solutions of our matrix-valued SDEs. And as an application 
of It\^o formula, another set of sufficient conditions such as monotone and local Lipschitz conditions are discussed for the existence and uniqueness of solutions of our matrix-valued SDEs.  

\section{Preliminaries}
\subsection{Motivation of matrix-valued SDEs}
To start with, we would like to set up the currency exchange market via matrix formulation. We assume that there are $n (\ge2)$ different currencies 
which are tradable in a currency exchange market. Set the (ordered) list of currencies by $\Lambda_0:=\{1,...,n\}$. 
Just envisage a table that we list all currencies in a row and in a column in order, so that all the pairing exchange rates (including a currency exchanges with itself) form a matrix. 
Let $\Lambda_1:=\{1,...,N\}$ be the trading dates. There are two kind of the matrices related to the currency markets as following 

\begin{itemize} 
	
	\item Foreign exchange rate $S^{(k)}$ and $S^{(k+1)}$ at date $k$ and date $k+1$ respectively,
	\begin{equation}\label{SK}
	S^{(k)}= \left(\begin{array}{ccccc}
	s_{11}^{(k)} & \bar s_{12}^{(k)}&\cdots&\bar s_{1n}^{(k)}\\
	\underline s_{21}^{(k)} & s_{22}^{(k)}&\cdots&\bar s_{2n}^{(k)}\\
	\vdots & \vdots &\ddots&\vdots\\
	\underline s_{n1}^{(k)} &\underline s_{n2}^{(k)}&\cdots&s_{nn}^{(k)}\\
	\end{array}
	\right) 
	\end{equation}
	and 
	\begin{equation}\label{SK+1}
	S^{(k+1)}= \left(\begin{array}{ccccc}
	s_{11}^{(k+1)} & \bar s_{12}^{(k+1)}&\cdots&\bar s_{1n}^{(k+1)}\\
	\underline s_{21}^{(k+1)} & s_{22}^{(k+1)}&\cdots&\bar s_{2n}^{(k+1)}\\
	\vdots & \vdots &\ddots&\vdots\\
	\underline s_{n1}^{(k+1)} &\underline s_{n2}^{(k+1)}&\cdots&s_{nn}^{(k+1)}\\
	\end{array}
	\right), 
	\end{equation}
where, for  each pair  $(i,j)\in\Lambda_0\times\Lambda_0$ with $i<j$, $\bar s_{ij}^{(\cdot)} $ denotes the currency exchange rate for the investor to buy the  foreign currency $j$ via the (home) currency $i$ at the date $k\in\Lambda_1$; and for  $i>j$, $\underline s_{ij}^{(\cdot)}$ stands for the currency exchange rate for the investor to sell the foreign currency $j$ to get 
the home currency $i$ at the date $k\in\Lambda_1$. The entries above the main diagonal are called the upper triangular part to show the bank quotations of selling and the entries below the main diagonal are named as the lower triangular part to show the bank quotations of buying.  When $i=j$ the value of $s_{ij}$ equals to one, as the currencies against itself is always one. 

\item With the foreign exchange market mechanism (showed in Figure), when trading currencies,  investors buy (or sell) foreign currencies at the date $k$ and then  sell (or buy) foreign currencies at the day $k+1$. Then we have following two combined matrices for investors to trade currencies at different time.
\begin{equation*} 
S_{(k+1)}^{(k)}= \left(\begin{array}{ccccc}
s_{11}^{(k)} & \bar s_{12}^{(k)}&\cdots&\bar s_{1n}^{(k)}\\
\underline s_{21}^{(k+1)} & s_{22}^{(k)}&\cdots&\bar s_{2n}^{(k)}\\
\vdots & \vdots &\ddots&\vdots\\
\underline s_{n1}^{(k+1)} &\underline s_{n2}^{(k+1)}&\cdots&s_{nn}^{(k)}\\
\end{array}
\right),
\end{equation*}
and 
\begin{equation*} 
S_{(k)}^{(k+1)}= \left(\begin{array}{ccccc}
s_{11}^{(k+1)} & \bar s_{12}^{(k+1)}&\cdots&\bar s_{1n}^{(k+1)}\\
\underline s_{21}^{(k)} & s_{22}^{(k)}&\cdots&\bar s_{2n}^{(k+1)}\\
\vdots & \vdots &\ddots&\vdots\\
\underline s_{n1}^{(k)} &\underline s_{n2}^{(k)}&\cdots&s_{nn}^{(k)}\\
\end{array}
\right),
\end{equation*}
\end{itemize} 

For more related matrices which linked to the real world transaction, the reader is refered to  \cite{RenWu}.

Take a real world observation, one can have the following graph for the pairing 
exchange rate for any chosen pair of daily trading currencies.      
    
\begin{center}
	\includegraphics[scale=0.96] {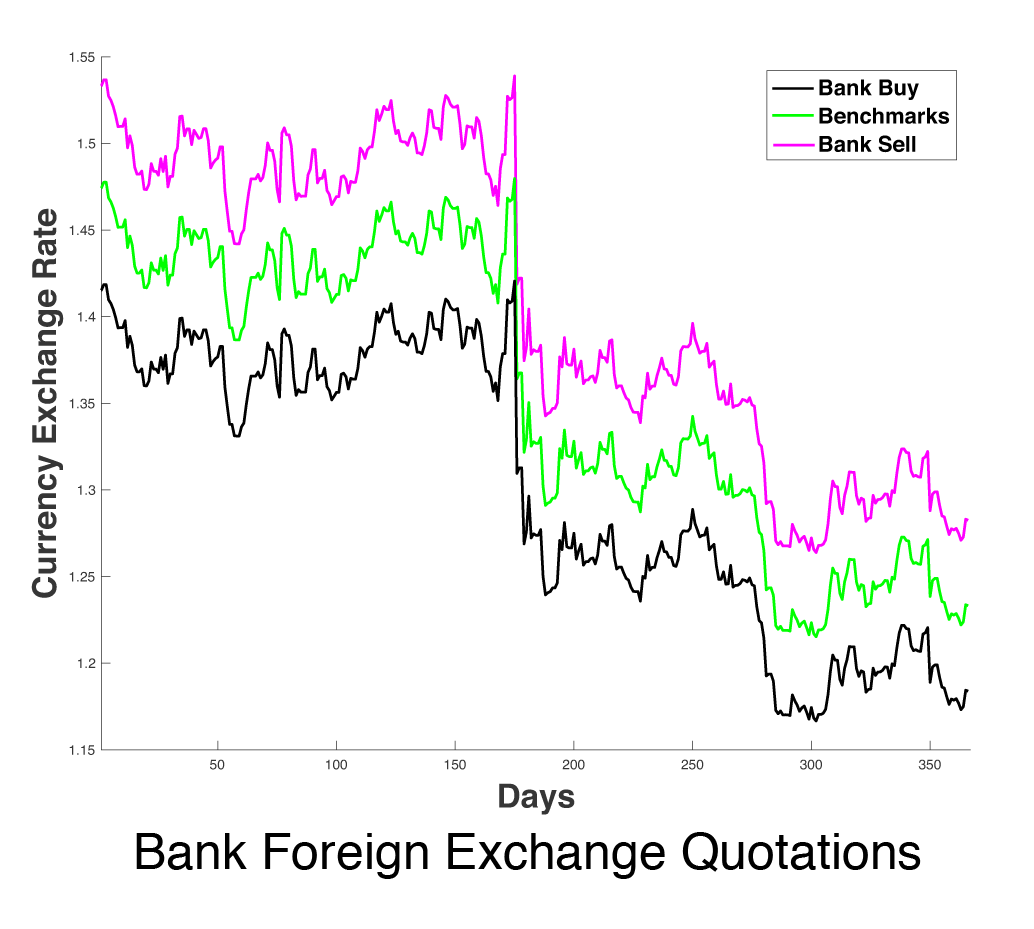}
\end{center}
This graph shows the foreign exchange rates for investor to buy (pink line) and sell (black line) foreign currencies. 

We observe that the Benchmark (green line) in the graph is volatile so it consists the mean value part of the foreign exchange rates for buying and selling currencies and 
the volatile factors caused by uncertainty of the currency market, in which the volatility of the pink line and black line are followed proportionally by the green line. This 
suggests that the continuous time modelling should involve SDEs. Moreover, envisage the table of currency exchange rates, one ends up 
naturally with matrix-valued SDEs.  

The value in the figure we can explicate in the matrices, then we can define a matrix-valued  SDE as following 
$\mathbb{R}^{n}\otimes \mathbb{R}^{n}$-valued SDE
\begin{equation*}
dS(t)=b(t,S(t))dt+\sigma(t,S(t))dB_t ,~~~~t>0,~~~~~S_0=s\in
\mathbb{R}^{n}\otimes \mathbb{R}^{n},
\end{equation*}
where $b,\sigma:\mathbb{R}_+\times \mathbb{R}^{n}\otimes
\mathbb{R}^{n}\rightarrow \mathbb{R}^{n}\otimes \mathbb{R}^{n}$ and
$(B_t)_{t\geq 0}$ is an $\mathbb{R}^{n}\otimes
\mathbb{R}^{n}$-valued Brownian motion we will define it later. Note that,  $b(t,S(t))dt$ can be represented as the macro-index such as consumer price index, inflation rate, interest rate, 
taxation, etc., which the mean value of the currency exchange rates and $\sigma(t,S(t))dB_t$ can be represented as the micro-index such as volume of transaction per day, transaction cost, etc., which influence the range of the high fluctuation based on the mean value of the currency exchange rates.

\subsection{$\mathbb{R}^{n}\otimes \mathbb{R}^{n}$-valued Brownian motion}
For arbitrarily fixed positive integer $n\ge2$, we let $\mathbb{R}^{n}$ be the $n$-dimensional Euclidean space
with the inner product $\langle\cdot , \cdot\rangle$ which induces the norm $|\cdot|$ and $\mathbb{R}^{n}\otimes \mathbb{R}^{n}$ the totality of all $n\times n$-matrices with real 
entries. For $A=(a_{ij})_{n\times n}\in \mathbb{R}^{n}\otimes \mathbb{R}^{n}$, $\mbox{ trace(A)}$ means the trace of $A$, i.e., $\mbox{trace(A)}=\sum_{i=1}^na_{ii}$.  Set
\begin{equation}\label{Huaxianzi0}
\langle A,B \rangle_{HS}:=\mbox{trace}(A^{T}B),~~~~~~~A, B \in \mathbb{R}^{n}\otimes \mathbb{R}^{n},
\end{equation}
where $A^{T}$ signifies  the transpose of $A$. 
 For any $A,B,C \in \mathbb{R}^{n}\otimes \mathbb{R}^{n}$ and $\lambda\in\R$,   a simple calculation shows that 
 \begin{itemize}
 \item Positive-definiteness: $\langle A,A\rangle_{HS}\geq 0$ and $\langle A,A\rangle_{HS}=0\Leftrightarrow A=0$;
 \item Conjugate symmetry: $\langle A,B\rangle_{HS}={\langle B,A\rangle}_{HS}$;
 \item Linearity: $\langle \lambda A,B\rangle_{HS}=\lambda \langle A,B\rangle_{HS}$ and $\langle A+B,C\rangle_{HS}=\langle A,C\rangle_{HS}+\langle B,C\rangle_{HS}$.
 \end{itemize} 
So $(\mathbb{R}^n\otimes \mathbb{R}^n, \langle\cdot,\cdot\rangle_{HS})$ is a real ($n^2$-dimensional) Hilbert space whose inner product induces the Hilbert Schmidt norm 
$\|\cdot\|_{HS}$. For $1\le i,j\le n$, let $e_{ij}\in\R^n\otimes\R^n$, in which the entry  locates in the intersection of the $i$-th row and the $j$-th column is $1$ and the other
entries are equal to zero. Observe that  $\langle e_{ij},e_{kl}\rangle_{HS}=1$ for $i=k, j=l$,  otherwise $\langle e_{ij},e_{kl}\rangle_{HS}=0$. Hence 
 $e_{ij},i,j=1,2,...,n$, are linear independent and any $A\in\R^n\otimes\R^n$ can be represented by  the family $\{e_{ij}\}_{1\le i,j\le n}$. Therefore, 
$\{e_{ij}\}_{1\le i,j\le n}$ forms an orthogonal basis of $ \mathbb{R}^{n}\otimes \mathbb{R}^{n}$. For $1\le i,j\le n$,  let $\{B^{(ij)}_t,{t\ge0}\}_{i,j=1,2...,n}$ be a family  
of independent $\mathbb{R}$-valued Brownian motions defined on a given filtered probability space $(\Omega, \mathcal{F},(\mathcal {F}_t)_{t\ge0},\mathbb{P})$. Set
\begin{equation}\label{penzhai}
B_t:=\sum_{i,j=1}^{n}B^{(ij)}_te_{ij}, \quad t\ge0. 
\end{equation}
For any $x,y\in  \mathbb{R}^{n}\otimes \mathbb{R}^{n}$,  the tensor product between $x$ and $y$, which is a linear operator from $\mathbb{R}^{n}\otimes \mathbb{R}^{n}$ to 
$\mathbb{R}^{n}\otimes \mathbb{R}^{n}$, is defined via the Kronecker product of matrices, by 
\begin{equation*} \label{huajiang}
(x\otimes y)(z)=\langle x,z\rangle_{HS} y,~~~~z\in  \mathbb{R}^{n}\otimes \mathbb{R}^{n}.
\end{equation*}
The covariation matrix of $B_t$ is given by 
$$\mbox{ Cov}(B_t):=\E((B_t-\mathbb{E} B_t)\otimes (B_t-\mathbb{E} B_t)).$$
Noting that $\mathbb{E} B_t=0$, we then have from \eqref{penzhai} that, for any $u,v\in
\mathbb{R}^{n}\otimes \mathbb{R}^{n}$,
\begin{equation}\label{Belfast}
\begin{split}
 &\E\Big(\langle ((B_t-\mathbb{E} B_t)\otimes (B_t-\mathbb{E} B_t))u, v\rangle_{HS}\Big)\\
 &=\sum_{i,j=1}^{n}\sum_{k,l=1}^{n}\mathbb{E}\Big(B^{(ij)}_tB^{(kl)}_t\langle e_{ij}, u\rangle_{HS} \langle e_{kl}, v\rangle_{HS}\Big)\\
 &=\sum_{i,j=1}^{n}\mathbb{E}\Big((B^{(ij)}_t)^2\langle e_{ij}, u\rangle_{HS} \langle e_{ij}, v\rangle_{HS}\Big)\\
 &=t\sum_{i,j=1}^{n}\langle e_{ij}, u\rangle_{HS} \langle e_{ij}, v\rangle_{HS}\\
 &=t\langle u, v\rangle_{HS},\end{split}
\end{equation}
where in the second identity we have used the fact that $B^{(ij)}_t$ and $B^{(kl)}_t$ are mutually independent in the case $i\neq k$ or $j\neq l$. Therefore, we get that the covariation 
matrix of $B_t$ is $tI_{n\times n}$ with $I_{n\times n}$ being the $n\times n$ identical matrix. This clearly shows that $(B(t))_{t\ge0}$ is an $\R^n\otimes\R^n$-valued Brownian motion. 

\subsection{The It\^o stochastic integrals}
This subsection is devoted to the definition of the It\^o stochastic integrals against $\R^n\otimes\R^n$-valued Brownian motion, defined in \eqref{penzhai}. Let us fix $T>0$ arbitrarily. 
An $\R^n\otimes \R^n$-valued stochastic process $(V(t))_{t\ge0}$ is said to be {\it simple} if there exists a partition $0=t_0\leq t_1\leq \cdots \leq t_{k-1} \leq t_k=T$ of the interval $[0,T]$ 
such that
$$
V(t,\omega)=\sum_{i=0}^{k-1}\varphi_i(\omega)I_{[t_i,t_{i+1})}(t), \quad t\in[0,T],
$$
where the $\R^n\otimes \R^n$-valued random variable $\varphi_i(\omega)$ is $\mathcal{F}_{t_i}$-measurable, $i=0,1,\cdots,k-1$. For the $\R^n\otimes \R^n$-valued simple process 
$(V(t))_{t\ge0}$, the It\^o stochastic integral of $(V(t))_{t\ge0}$ with respect to $(B(t))_{t>0}$ is defined by 
$$
\int_{0}^{T}V(t)dB(t)=\sum_{i=0}^{k-1}\varphi_i(B(t_{i+1})-B(t_i)).
$$
Next, let $\mu^2_T(\mathbb{R}^n\otimes \mathbb{R}^{n})$ be the set of all $(\mathcal{F}_t)_{t\ge0}$-adapted stochastic processes $V:\mathbb{R}_+\rightarrow \mathbb{R}^n\otimes \mathbb{R}^{n}$ such that 
$
\int_{0}^{T}\E\lVert V(t)\rVert^2_{HS}dt<\infty. 
$
To define the It\^o stochastic integral for $V\in\mu^2_T(\mathbb{R}^n\otimes \mathbb{R}^{n})$ with respect to $(B(t))_{t\ge0}$, we need the following approximation.
\begin{lemma}\label{app}
If $V\in\mu^2_T(\mathbb{R}^n\otimes \mathbb{R}^{n})$, there exists a sequence of simple stochastic processes $(V_k(t))_{k\in\mathbb{N}}$ such that 
\begin{equation}\label{app1}
\lim_{k\rightarrow\infty}\int_0^T\E\|V(t)-V_k(t)\|_{HS}^2 dt =0.
\end{equation}
\end{lemma}
\begin{proof} 
The proof is divided into three steps. 

\noindent{{\bf Step 1:} 
If $V\in\mu^2_T(\mathbb{R}^n\otimes \mathbb{R}^{n})$ is bounded and continuous}, we then set 
$$
V_k(t):=\sum_{i=0}^{2^k-1}V(t_i)I_{[t_i,t_{i+1})}(t)+V(t_{2^k})I_{[t_{2^k},T]}(t), \quad k\in\mathbb{N},
$$
where $t_i=\frac{i[T]}{2^k}, i=0,1,2,...,2^k$, with $[T]$ being the integer part of $T$, and clearly $0=t_0<t_1<t_2<\cdots<t_{2^k-1}<t_{2^k}=[T]\le T$ forms a partition of the interval $[0,T]$. 
By the continuity of $V, V_k(t)\rightarrow V(t)$ as $k\to\infty$. This, together with the dominated convergence theorem, leads to \eqref{app1}.

\noindent{{\bf Step 2:} Let $V\in\mu^2_T(\mathbb{R}^n\otimes \mathbb{R}^{n}) $ be bounded.}  For each $k\ge1$, let $\phi_k\in C^2(\mathbb{R})$ such that 
$\phi_k(x)=0$ for $|x|\geq \frac{1}{k}$ and $\int_{-\infty}^{\infty}\phi_k(x)dx=1$. Set 
$$
	V_k(t):=\int_{-\infty}^{\infty}\phi_k(s-t)V(s)ds.
$$
	We remark that $V_k(t)$ is bounded and continuous due to the fact that $\phi_k$ possesses compact support and that $V$ is uniformly bounded. Moreover, according 
	to the property of mollifier, we get that 
	$$
 \int_{0}^{T}	\E\lVert V(t)-V_k(t)\rVert^2_{HS}ds\rightarrow0~~~~\mbox{as}~~k\rightarrow\infty.
	$$
	\noindent{{\bf Step 3:} For $V\in\mu^2_T(\mathbb{R}^n\otimes \mathbb{R}^{n})$ being unbounded}, let 
	\begin{equation*}
	V_k(t)=\begin{cases}
	kI_{n\times n}~~~~~~~~~~~~~~~\lVert V(t)\rVert_{HS}\geq k\\
	V(t)~~~~~~~~~~~~~~~~~ \lVert V(t)\rVert_{HS}\leq k.
	\end{cases}
	\end{equation*}
A straightforward calculation shows that
	$$
 \int_{0}^{T}	\E\lVert V(t)-V_k(t)\rVert^2_{HS}ds\rightarrow0~~~~\mbox{as}~~k\rightarrow\infty.
   $$
	This completes the proof. 
\end{proof}

With the help of Lemma \ref{app}, one can define the It\^o stochastic integral for $V\in\mu^2_T(\mathbb{R}^n\otimes \mathbb{R}^{n})$ as follows 
\begin{equation}\label{app2}
\int^T_0V(t)dB(t)=\lim_{k\to\infty}\int^T_0V_k(t)dB(t)~~~~~~\mbox { in } L^2(\Omega; (\mathbb{R}^n\otimes \mathbb{R}^n, \langle\cdot,\cdot\rangle_{HS})),
\end{equation}
i.e., 
\begin{equation}\label{eqeq}
\lim_{k\rightarrow \infty}\E\Big\|\int^T_0V(t)dB(t)-\int^T_0V_k(t)dB(t)\Big\|_{HS}^2=0,
\end{equation}
for  any sequence of simple stochastic processes $(V_k(t))_{k\in\mathbb{N}}$ such that 
$$\lim_{k\rightarrow\infty}\int_0^T\E\|V(t)-V_k(t)\|_{HS}^2 dt =0. $$

\subsection{The quadratic variation process of stochastic integrals} 
Before presenting the quadratic variation process of the It\^o stochastic integral $\int_0^t V(s)d B(s)$, $t\ge0$, defined in \eqref{app2}, we need to introduce some notions. 
\begin{itemize}
	\item An $n\times n$ matrix $A$ is nonnegative if 	$\< Aa, a\>_{HS}>0$ for any $a\in \mathbb{R}^n\otimes \mathbb{R}^{n}.$
	\item $A: t\mapsto A(t)\in \mathbb{R}^{n}\otimes \mathbb{R}^{n}$ is non-decreasing, if $A(t)-A(s)$ is nonnegative for $t\geq s\ge0$. In this case, we then write $A(t)\geq A(s),~t\geq s\ge0$. 
	\item For $M\in\mu^2_T(\mathbb{R}^n\otimes \mathbb{R}^{n})$ such that $\< M(t), a\>_{HS}$ for $a\in\mathbb{R}^{n}\otimes \mathbb{R}^{n}$ is an $\mathbb{R}$-valued $\mathcal{F}_t$-martingale, if there exists a nonnegative and non-decreasing $A(t)\in \mathbb{R}^n\otimes \mathbb{R}^{n}$ such that, 
	for any $a,b\in \mathbb{R}^{n}\otimes \mathbb{R}^{n}$,
	$$
	\< M(t), a\>_{HS}\< M(t), b\>_{HS}-\< A(t)a, b\>_{HS}
	$$
	is an $\mathbb{R}$-valued $\mathcal{F}_t$-martingale, then $A(t)\in \mathbb{R}^{n}\otimes \mathbb{R}^{n}$ is called the quadratic variation process of $M(t)$, denoted by $\ll M, M \gg(t)$. 
\end{itemize}

\begin{lemma}
	For $V\in\mu^2_T(\mathbb{R}^n\otimes \mathbb{R}^{n})$, let $M(t)=\int_0^tV(s)d B(s)$. Then, for any $a\in \mathbb{R}^{n}\otimes \mathbb{R}^{n}$, $(\<M(t),a\>_{HS})_{t\ge0}$ is an $\mathbb{R}$-valued $(\mathcal{F}_t)_{t\ge0}$-martingale, and moreover 
	\begin{equation}\label{app9}
	\ll M, M \gg(t)=\int_{0}^{t}V(s)V^T(s)ds.
	\end{equation}
\end{lemma}
\begin{proof} The first statement that,  for any $a\in \mathbb{R}^{n}\otimes \mathbb{R}^{n}$, $(\<M(t),a\>_{HS})_{t\ge0}$ is an $\mathbb{R}$-valued $(\mathcal{F}_t)_{t\ge0}$-martingale 
	 is routine. Let us show the second statement. 
	For any $a,b\in \mathbb{R}^{n}\otimes \mathbb{R}^{n}$ and $t\ge s$, 
	it follows	from \eqref{penzhai} that   
	\begin{equation}\label{niu1}
	\begin{split}
	&	\E\Big(\<M(t), a\>_{HS}\<M(t), b\>_{HS}\Big|\mathcal{F}_s\Big)\\
	&=\sum_{i,j=1}^{n}\sum_{k,l=1}^{n}\E\bigg(\int_{s}^{t}\<V(u)e_{ij}, a\>_{HS}dB^{(ij)}_u\int_{s}^{t}\<V(u)e_{kl}, b\>_{HS}dB^{(kl)}_u\Big|\mathcal{F}_s\bigg)\\
	&=\sum_{i,j=1}^{n}\E\bigg(\int_{s}^{t}\<V(u)e_{ij}, a\>_{HS}\<V(u)e_{ij}, b\>_{HS}du\Big|\mathcal{F}_s\bigg)\\
	&=\sum_{i,j=1}^{n}\E\bigg(\int_{s}^{t}\<e_{ij}, V^T(u)a\>_{HS}\<e_{ij}, V^T(u)b\>_{HS}du\Big|\mathcal{F}_s\bigg)\\
	&=\Big\<\E\Big(\int_{s}^{t}V(u)V^T(s)du\Big)\Big|\mathcal{F}_s\Big)a, b\Big\>_{HS}.
	\end{split}
	\end{equation}
	Next we aim to show that, for any $a,b\in\R^n$
	\begin{equation}\label{niu2}
	\Gamma(t):=\<M(t), a\>_{HS}\<M(t), b\>_{HS}-\Big\<\Big(\int_{0}^{t}V(s)V^T(s)ds\Big)a, b\Big\>_{HS}
	\end{equation}
	is a martingale, i.e.,    for any $t\geq s$,
	$
	\E(\Gamma(t)|\mathcal{F}_s)=\Gamma(s).
	$
	For notation brevity, we set $\Lambda(s,t):=\int_s^tV(u)d B(u)$ for $s\in[0,t]$. Note that $\Lambda(0,t)=M(t)$.
	For any $t\geq s$,
	\begin{align*}\label{niu4}
	E(\Gamma(t)|\mathcal{F}_s)
	&=\Gamma(s)+\E\bigg(\Big(\<M(s), a\>_{HS}\<\Lambda(s,t), b\>_{HS}+\<\Lambda(s,t), a\>_{HS}\<M(s), b\>_{HS}\\
	&\quad+\<\Lambda(s,t) ,a\>_{HS}\<\Lambda(s,t), b\>_{HS}-\int_{s}^{t}\<V(u)V^T(u)a, b\>_{HS}du\Big)\Big| \mathcal{F}_s\bigg)\\
	&=\Gamma(s)+\<M(s), a\>_{HS}\E\Big(\<\Lambda(s,t), b\>_{HS}\Big| \mathcal{F}_s\Big)\\
	&\quad+\E\Big(\<\Lambda(s,t), a\>_{HS}\Big|  \mathcal{F}_s\Big)\<M(s), b\>_{HS}\\
	&\quad+\E\Big(\<\Lambda(s,t), a\>_{HS}\<\Lambda(s,t), b\>_{HS}\Big|  \mathcal{F}_s\Big)\\
	&\quad-\E\bigg(\int_{s}^{t}\<Vu)V^T(u)a, b\>_{HS}du\Big| \mathcal{F}_s\bigg)\\
	&=\Gamma(s)+\E\Big(\<\Lambda(s,t), a\>_{HS}\<\Lambda(s,t), b\>_{HS}\Big|\mathcal{F}_s\Big)\\
	&\quad-\E\Big(\int_{s}^{t}\<V(u)V^T(u)a, b\>_{HS}du\Big|\mathcal{F}_s\Big)\\
	&=\Gamma(s),
	\end{align*}
	where in the last display we have used \eqref{niu1}. Whence, \eqref{app9} holds. 
	The proof is therefore complete. 
\end{proof}

\subsection{It\^o isometry for stochastic integrals}
In this subsection, we aim to establish the It\^o isometry for stochastic integral established. 

\begin{lemma}\label{Tiger}
	{\rm For $V\in\mu^2_T(\mathbb{R}^n\otimes \mathbb{R}^{n})$,
		\begin{equation}\label{RR}
		\E\Big\|\int_0^TV(t)d B(t)\Big\|_{HS}^2=n\int_0^T\E\|V(t)\|_{HS}^2d
		t.
		\end{equation}
	}
\end{lemma}

\begin{proof}
	Firstly, we assume that $(V(t))_{t\in[0,T]}$ is a simple stochastic process in the form $V(t,\omega)=\sum_{k=0}^{m-1}h_k(\omega)I_{[t_k,t_{k+1})}$, where $h_k(\omega)\in \mathcal{F}_{t_k}$ and $0=t_0< t_1< t_2<\cdots < t_m=T$ is a partition of the interval $[0,T]$. Note that  
	\begin{equation*}
	\begin{split}
	\E \Big\lVert  \int_{0}^{T} V(t)dB(t)\Big\lVert _{HS}^2&=\E
	\Big\lVert  \sum_{k=0}^{m-1}h_k\triangle B_k\Big\lVert _{HS}^2\\
	&= \sum_{k=0}^{m-1}\E\|h_k\triangle B_k\|_{HS}^2+\sum_{k\neq
		l}\E\langle h_k\triangle B_k,h_l\triangle B_l\rangle_{HS}\\
	&=:I_1+I_2,
	\end{split}
	\end{equation*}
	where 
	$\triangle B_k:=B(t_{k+1})-B(t_k)$.   
	From \eqref{Huaxianzi0} and  \eqref{Belfast},
	we deduce that
	\begin{align*}
	I_1&=\sum_{k=0}^{m-1}\sum_{i,j=1}^n\E(\langle \triangle
	B_k,h^T_ke_{ij}\rangle_{HS}\langle \triangle
	B_k,h^T_ke_{ij}\rangle_{HS})\\
	&=\sum_{k=0}^{m-1}(t_{k+1}-t_k)\sum_{i,j=1}^n\E\langle
	h^T_ke_{ij},h^T_ke_{ij}\rangle_{HS}\\
	&=\sum_{k=0}^{m-1}(t_{k+1}-t_k)\sum_{i,j=1}^n\E\langle
	h_kh^T_ke_{ij},e_{ij}\rangle_{HS}\\
	&=n\sum_{k=0}^{m-1}(t_{k+1}-t_k)\E\,\mbox{trace}(h^T_kh_k)\\
	&=n\sum_{k=0}^{m-1}(t_{k+1}-t_k)\E\|h_k\|_{HS}^2\\
	&=n\int_0^T\E\|V(t)\|_{HS}^2d t.
	\end{align*}
	Next, without loss of generality, we assume $k>l.$ Then,
	\begin{equation*}
	\begin{split}
	I_2&=\sum_{i,j=1}^n\E(\langle \triangle
	B_k,h^T_ke_{ij}\rangle_{HS}\langle \triangle
	B_l,h^T_le_{ij}\rangle_{HS})\\
	&=\sum_{i,j=1}^n\E(\E(\langle \triangle
	B_k,h^T_ke_{ij}\rangle_{HS}\langle \triangle
	B_l,h^T_le_{ij}\rangle_{HS})|\mathcal{F}_{t_l})\\
	&=\sum_{i,j=1}^n\E(\langle \triangle
	B_l,h^T_le_{ij}\rangle_{HS}\E(\langle \triangle
	B_k,h^T_ke_{ij}\rangle_{HS}|\mathcal{F}_{t_l}))\\
	&=0,
	\end{split}
	\end{equation*}
	where in the second identity we have used the tower property of conditional expectation and in the last two identity we have utilized that $\langle\triangle B_l,h^T_le_{ij}\rangle_{HS}$ is adapted to $\mathcal{F}_{t_l}$ and that, for any  $A\in\mathcal{F}_{t_l}$, 
	\begin{equation*}
	\begin{split}
\E(\langle \triangle
B_k,h^T_ke_{ij}\rangle_{HS}{\bf1}_A)&=\E(\E(\langle \triangle
B_k,h^T_ke_{ij}\rangle_{HS}{\bf1}_A|\mathcal{F}_{t_k}))\\
&=\E(({\bf1}_A\langle h^T_ke_{ij},\E(\triangle
B_k|\mathcal{F}_{t_k})\rangle_{HS})\\
&=\E(({\bf1}_A\langle h^T_ke_{ij},\E\triangle
B_k\rangle_{HS})\\
&=0,
	\end{split}
	\end{equation*} 
	in which in the first identity  we have utilized the tower property of conditional expectation, in the second identity we have used ${\bf1}_Ah_k\in\mathcal{F}_{t_k}$ and the third indentity is due to $\triangle
	B_k$ is independent of $\mathcal{F}_{t_k}$.
	So, \eqref{RR} holds. If  $V\in\mu^2_T(\mathbb{R}^n\otimes \mathbb{R}^{n})$, then, according to Lemma \ref{app},  there exists a sequence of simple   stochastic processes $(V_k(t))_{t\ge0}$ such that 
	\begin{equation*} 
	\lim_{k\rightarrow\infty}\int_0^T\E\|V(t)-V_k(t)\|_{HS}^2 dt =0.
	\end{equation*}
	This, in addition to 
	\begin{equation*}
	\E\Big\|\int_0^TV_k(t)d B(t)\Big\|_{HS}^2=n\int_0^T\E\|V_k(t)\|_{HS}^2d
	t
	\end{equation*}
	due to Lemma \ref{Tiger} and, by virtue of \eqref{eqeq},
	$$
	\lim_{k\rightarrow\infty}\Big\|\int_{0}^{T}V(s)dB(s)-\int_{0}^{T}V_k(s)dB(s)\Big\|_{HS}^2=0
	$$
	gives the desired assertion. 
\end{proof}

\section{Existence and uniqueness for  $\mathbb{R}^{n}\otimes \mathbb{R}^{n}$-valued SDEs} 

We are concerned with the following $\mathbb{R}^{n}\otimes \mathbb{R}^{n}$-valued SDE  
\begin{equation}\label{Belfast1}
dX(t)=b(t,X(t))dt+\sigma(t,X(t))dB_t ,~~~~t>0,~~~~~X_0=x\in
\mathbb{R}^{n}\otimes \mathbb{R}^{n},
\end{equation}
where $b,\sigma:\mathbb{R}_+\times \mathbb{R}^{n}\otimes
\mathbb{R}^{n}\rightarrow \mathbb{R}^{n}\otimes \mathbb{R}^{n}$ are jointly measurable and
$(B_t)_{t\geq 0}$ is an $\mathbb{R}^{n}\otimes
\mathbb{R}^{n}$-valued Brownian motion  defined as in
\eqref{penzhai}.

For the fixed   time horizontal $T>0$, we assume that there exist integrable functions $\kk_1,\kk_2:
[0,T]\rightarrow\mathbb{R}_{+}$, such that 
\begin{enumerate}
	\item[({\bf A1})]
	$ \lVert
	b(t,x)-b(t,y)\rVert^2_{HS}+\lVert\sigma(t,x)-\sigma(t,y)\rVert^2_{HS}
	\leq \kk_1(t)\lVert x-y\rVert^2_{HS}, $
	\item[({\bf A2})] $\|b(t,0)\|_{HS}^2+\|\sigma(t,0)\|_{HS}^2\le\kk_2(t).$
\end{enumerate}

It is clear from ({\bf A1}) and ({\bf A2}) that

\begin{equation}\label{66}
\begin{split}
&\lVert
b(t,x)\rVert^2_{HS}+\lVert\sigma(t,x)\rVert^2_{HS}\\&\le2\{\lVert
b(t,0)\rVert^2_{HS} +\lVert\sigma(t,0)\rVert^2_{HS}\\
&\quad+\lVert
b(t,x)-b(t,0)\rVert^2_{HS}+\lVert\sigma(t,x)-\sigma(t,0)\rVert^2_{HS}\}\\
&\le 2\kk(t)(1+\|x\|_{HS}^2),
\end{split}
\end{equation}
where $\kk(t):=\kk_1(t)+\kk_2(t).$ Therefore, under ({\bf A1}) and ({\bf A2}), $b$ and $\sigma$ are of linear growth.

\begin{theorem}\label{push}
	Under $({\bf A1}) $ and $({\bf A2})$, \eqref{Belfast1} has a unique global strong solution $(X(t))_{t\ge0}$, 
	which satisfies that
	\begin{equation}
	\E\Big( \sup_{0\leq t\leq T}\lVert X(t)\rVert^2_{HS} \Big)\leq(1+3\lVert x\rVert^2_{HS}) e^{6(T+4n)\int _{0}^{T}\kk(s)ds}.
	\end{equation}
	
\end{theorem}

\begin{proof}
	The proof on existence and uniqueness of solutions to \eqref{Belfast1} is based on the classical Banach fixed point theorem,
		cf. e.g. \cite{TrumanWu}.  To end this,  we define the space
	\begin{equation*}
	\begin{split}
	\H_2=\Big\{u|&u:\R_+\rightarrow\R^n\otimes\R^n \mbox{ is predictable
		and such that }\\
	& \3u\3_{2}:= \sup_{t\in[0,T]} \Big(\mathbb{E}\lVert
	u(t)\rVert_{HS}^2\Big)^{\frac{1}{2}}<\infty\Big\},
	\end{split}
	\end{equation*}
	which is a Banach space.  Define the mapping $\Gamma:
	\H_2\rightarrow\R^n\otimes\R^n$ by
	\begin{equation*}
	\Gamma (u)(t)=x+\int _{0}^{t}b(s,u(s))ds+\int
	_{0}^{t}\sigma(s,u(s))dB(s),~~~~t>0,~~u\in\H_2.
	\end{equation*}
	One can show that $\Gamma: \H_2\rightarrow\H_2$ is contractive whenever $T>0$ is sufficiently small.  To overcome the drawback of the approach above,  for any $\lambda>0$ to be determined, we redefine the following space:
	\begin{equation*}
	\begin{split}
	\H_{2,\lambda}=\Big\{u|&u:\R_+\rightarrow\R^n\otimes\R^n \mbox{ is
		predictable
		and such that }\\
	& \3u\3_{2,\lambda}:=\sup_{t\in[0,T]} \Big(e^{-\lambda
		t}\Big(\mathbb{E}\lVert
	u(t)\rVert_{HS}^2\Big)^{\frac{1}{2}}\Big)<\infty\Big\},
	\end{split}
	\end{equation*}
	which is also a Banach space. Define the mapping $\Gamma:
	\H_{2,\lambda}\rightarrow\R^n\otimes\R^n$ by
	\begin{equation}\label{Pan}
	\Gamma (u)(t)=x+\int _{0}^{t}b(s,u(s))ds+\int
	_{0}^{t}\sigma(s,u(s))dB(s),~~~~u\in\H_{2,\lambda}.
	\end{equation}
	In the sequel, we show that, for $\lambda>0$ sufficiently large,  the mapping $\Gamma$ defined above possesses the following properties:
	\begin{enumerate}
		\item [i) ]$\G$ maps $\H_{2,\lambda}$ to $\H_{2,\lambda}$; 
		\item [ii)]$\3\G(u_1)-\G(u_2)\3_{2,\lambda} \leq \alpha \3u_1-u_2\3_{2,\lambda} $ for some $\alpha \in (0,1)$.
	\end{enumerate}
	If i) and ii) hold, respectively, then, according to the Banach fixed point theorem, the mapping  $\G:\H_{2,\lambda} \rightarrow \H_{2,\lambda}$ admits a unique fixed point $u\in \H_{2,\lambda}$, i.e., $\G(u)(t)=u(t), t\in[0,T].$ This, together with \eqref{Pan}, yields that \eqref{Belfast1} has a global unique strong solution $(X(t))_{t\ge0}.$ So, it remains to show that i) and ii) hold, step-by-step. We start to show the property  i). For any $u \in \H_{2,\lambda}$,  we deduce from \eqref{Pan} that
	\begin{equation}\label{Faraday}
	\begin{split}
	&\3\G(u)\3_{2,\lambda}^2\\&= \sup_{t\in[0,T]} (e^{-2\lambda t}\mathbb{E}\lVert \G (u)(t)\rVert_{HS}^2)\\
	&=\sup_{t\in[0,T]} \Big\{e^{-2\lambda t} \E\Big\lVert x+\int _{0}^{t}b(s,u(s))ds +\int _{0}^{t}\sigma(s,u(s))dB(s)
	\Big\rVert_{HS}^2    \Big\}\\
	&\le3\sup_{t\in[0,T]} (e^{-2\lambda t} \lVert x\rVert_{HS}^2)+3\sup_{t\in[0,T]}\Big(e^{-2\lambda t} \E \Big\lVert\int _{0}^{t}b(s,u(s))ds \Big\rVert_{HS}^2 \Big)\\
	&\quad+3\sup_{t\in[0,T]} \Big(e^{-2\lambda t} \E \Big\lVert \int _{0}^{t}\sigma(s,u(s))dB(s)\Big\rVert_{HS}^2 \Big)\\
	&=:I_1+I_2+I_3.
	\end{split}
	\end{equation}
	For the term $I_1$,  it follows that
	\begin{equation}\label{Faraday0}
	I_1 \leq 3\lVert x\rVert_{HS}^2 < \infty.
	\end{equation}
	With regard to  $I_2$, applying H\"older's inequality	and \eqref{66},  we derive  that
\begin{equation}\label{panpan}
\begin{split}
I_2 
&\leq 3\sup_{t\in[0,T]}\Big(e^{-2\lambda t} \int _{0}^{t}e^{2\lambda s}ds\E\int _{0}^{t}e^{-2\lambda s}\lVert b(s,u(s)) \rVert_{HS}^2ds\Big)\\
&\leq \frac{3}{2\lambda} \int _{0}^Te^{-2\lambda t}\E \lVert b(tu(t) \rVert_{HS}^2dt\\
&\leq \frac{3}{\lambda}\int _{0}^{T}e^{-2\lambda t}\kk(t)(1+\E \lVert u(t)\rVert_{HS}^2dt\\
&\leq\frac{3}{\lambda} (1+\3u\3_{2,\lambda}^2)\int
_{0}^{T}\kk(t)dt.
\end{split}
\end{equation}
Concerning  $I_3$ , by  Lemma \ref{Tiger}, and \eqref{66}, we infer that 
\begin{equation}\label{Faraday2}
\begin{split}
I_3 &=3n\sup_{t\in[0,T]} \Big(e^{-2\lambda t} \int _{0}^{t}\E \lVert \sigma(s,u(s)) \rVert_{HS}^2ds \Big)\\
&\leq 6n\sup_{t\in[0,T]}\Big(e^{-2\lambda t} \int _{0}^{t}\kk(s)(1+\E\lVert u(s) \rVert_{HS}^2ds\Big)\\
&\leq 6n\sup_{t\in[0,T]}\Big(\int _{0}^{t}\kk(s)e^{-2\lambda s} (1+\E\lVert u(s) \rVert_{HS}^2ds\Big)\\
&\leq6n(1+\3u\3_{2,\lambda}^2)\int _{0}^{T}\kk(t)dt.
\end{split}
\end{equation}
By means of $u\in \H_{2,\lambda}$ and $\int_{0}^{T}\kk(t)dt<\infty$, we therefore obtain from \eqref{Faraday}-\eqref{Faraday2} that i) holds. Now, we aim to verify that   ii) is  also satisfied.  For any $u_1,u_2 \in \H_2$, it follows from \eqref{Pan} that
\begin{align*}
\3\G(u_1)-\G(u_2)\3_{2,\lambda}^2
&\leq 2\sup_{t\in[0,T]}\Big(e^{-2\lambda t}\E\Big\lVert\int _{0}^{t}(b(s,u_1(s))-b(s,u_2(s)))ds\Big\lVert_{HS}^2\Big)\\
&\quad+2\sup_{t\in[0,T]}\Big(e^{-2\lambda t}\E\Big\lVert\int_{0}^{t}(\sigma(s,u_1(s))-\sigma(s,u_2(s)))dB(s)\Big\lVert_{HS}^2\Big)\\
&=:J_1+J_2.
\end{align*}
Applying  H\"older's inequality and taking  {(\bf{A1})} into account, we get
\begin{equation}\label{Lib1}
\begin{split}
J_1
&\leq 2\sup_{t\in[0,T]}\Big(e^{-2\lambda t}\int _{0}^{t}e^{2\lambda s}ds\int _{0}^{t}e^{-2\lambda s}\E\lVert (b(s,u_1(s))-b(s,u_2(s)))\lVert_{HS}^2ds\Big)\\
&\leq\frac{1}{\lambda}\int_{0}^{T} e^{-2\lambda t}\kk_1(t)\E\lVert (u_1(s)-u_2(s)\lVert_{HS}^2dt\\
&\leq \frac{1}{\lambda}\int_{0}^{T}
\kk_1(t)dt\3u_1-u_2\3_{2,\lambda}^2.
\end{split}
\end{equation}
Next,  by Lemma \ref{Tiger} and {(\bf{A1})}, enable us to obtain that
\begin{equation}\label{Lib2}
\begin{split}
J_2&=2n\sup_{t\in[0,T]}\Big(e^{-2\lambda t}\int _{0}^{t}\E\lVert \sigma(s,u_1(s))-\sigma(s,u_2(s))\lVert_{HS}^2ds\Big)\\
&\leq 2n\sup_{t\in[0,T]}\Big(e^{-2\lambda t}\int _{0}^{t}\kk_1(s)\E\lVert u_1(s)-u_2(s)\lVert_{HS}^2ds\Big)\\
&=2n\sup_{t\in[0,T]}\Big(e^{-2\lambda t}\int _{0}^{t}\kk_1(s)e^{2\lambda s}e^{-2\lambda s}\E\lVert u_1(s)-u_2(s)\lVert_{HS}^2ds\Big)\\
&\leq 2n\3u_1-u_2\3_{2,\lambda}^2\sup_{t\in[0,T]}\Big(\int _{0}^{t}\kk(s)e^{-2\lambda(t-s)
}ds\Big).
\end{split}
\end{equation}
As a consequence,  taking \eqref{Lib1} and \eqref{Lib2} into consideration, we derive that
\begin{equation}\label{Lib3}
\begin{split}
\3\G(u_1)-\G(u_2)\3_{2,\lambda}^2&\leq \alpha\3u_1-u_2\3_{2,\lambda}^2,
\end{split}
\end{equation}
in which
\begin{equation*}
\alpha:=\frac{1}{\lambda}\int_{0}^{T}
\kk_1(t)dt+2n\sup_{t\in[0,T]}\Big(\int _{0}^{t}\kk(s)e^{-2\lambda(t-s)
}ds\Big).
\end{equation*}
Choosing $\lambda>0$ sufficiently large  so that $\alpha\in(0,1)$, we infer from \eqref{Lib3} that  ii) holds.

Let $t\in [0,T]$ be arbitrary. By the inequality: $(a+b+c)^2\leq3\,(a^2+b^2+c^2)$ for any $ a,b,c\in \mathbb R$,
\begin{equation}\label{Push}
\begin{split}
\E\Big(\sup_{0\leq s\leq t} \lVert X(s)\rVert^2_{HS}\Big)
&\leq 3\lVert x\lVert_{HS}^2+3\E\Big(\sup_{0\leq s\leq t}\Big\lVert\int _{0}^{s}b(r,X(r))dr\Big\lVert_{HS}^2\Big)\\
&\quad +3\E\Big(\sup_{0\leq s\leq t}\Big\lVert\int_{0}^{s}\sigma(r,X(r))dB(r)\Big\lVert_{HS}^2\Big)\\
&=:3\lVert x\lVert_{HS}^2+\Xi_1(t)+\Xi_2(t).
\end{split}
\end{equation}
For the term $\Xi_1(t)$,  the H\"older inequality enables us to give that
\begin{equation}\label{Push1}
\Xi_1(t)\leq 3t\int_{0}^{t}\E\lVert b(s,X(s))\lVert_{HS}^2ds.
\end{equation}
As far as the term $\Xi_2(t)$ is concerned,  using Doob's sub-martingale inequality followed by Lemma \ref{Tiger}, we obtain that
\begin{equation}\label{Push2}
\begin{split}
\Xi_2(t)&\leq 12n\E\Big\lVert \int_{0}^{t}\sigma(s,X(s))dB(s)\Big\lVert_{HS}^2\\
&=12n\int_{0}^{t}\E\lVert \sigma(s,X(s))\lVert_{HS}^2ds.
\end{split}
\end{equation}
Putting   \eqref{Push1} and \eqref{Push2} into \eqref{Push} and applying \eqref{66} yields that
\begin{equation*}
\begin{split}
&\E\Big( \sup_{0\leq s\leq t} \lVert X(s)\rVert^2_{HS} \Big)\\
&\leq3\lVert x\lVert_{HS}^2+3(t+4n)\int_{0}^{t} \Big\{\E\lVert b(s,X(s))\lVert_{HS}^2+\E\lVert \sigma(s,X(s))\lVert_{HS}^2 \Big\}ds\\
&\leq3\lVert x\lVert_{HS}^2+6(t+4n)\int_{0}^{t}\kk(s)(1+\E\lVert X(s)\lVert_{HS}^2)ds.
\end{split}
\end{equation*}
This further implies that
\begin{equation*}\label{Push4}
\begin{split}
& 1+\E\Big( \sup_{0\leq s\leq t} \lVert X(s)\rVert^2_{HS} \Big)\\ 
&\leq1+3\lVert x\lVert_{HS}^2+6(t+4n)\int_{0}^{t}\kk(s)\Big\{1+\E\Big(\sup_{0\leq r\leq s}\lVert X(r)\rVert^2_{HS}\Big)\Big\}ds.
\end{split}
\end{equation*}
Thus, by  Gronwall's inequality one has
\begin{equation*}
1+\E\Big( \sup_{0\leq s\leq t} \lVert X(s)\rVert^2_{HS} \Big) \leq(1+3\lVert
x\lVert_{HS}^2)e^{6(t+4n)\int_{0}^{t}\kk(s)ds}.
\end{equation*}
Consequently, the proof is  complete by taking $t=T$.
\end{proof}

\section{The It\^o formula for $\R^n\otimes\R^n$-valued SDEs and an apllication}

In this section, we intend to derive an It\^o formula based on the solution of matrix valued SDE.  
Before proceeding, let us introduce gradient operator and second-order gradient operator for $V\in C^2(\R^n\otimes\R^n;\R)$, 
the set of all twice Fr\'{e}chet differentiable functions defined on $\R^n\otimes\R^n$ and taking value in $\R$. The  gradient of $V$, 
denoted by $\nabla V$, is defined by, for  $X=(X_{ij})_{n\times n}\in\R^n\otimes\R^n$,
\begin{equation}\label{Ren}
\nabla V(X)= \left(\begin{array}{cccc}
\frac{\partial}{\partial X_{11}}V(X)& \frac{\partial}{\partial X_{12}}V(X) &  \cdots & \frac{\partial}{\partial X_{1n}}V(X)\\
\frac{\partial}{\partial X_{21}}V(X)& \frac{\partial}{\partial X_{22}}V(X) &  \cdots & \frac{\partial}{\partial X_{2n}}V(X)\\
\cdots& \cdots &  \cdots &\cdots\\
\frac{\partial}{\partial X_{n1}}V(X)& \frac{\partial}{\partial X_{n2}}V(X) &  \cdots & \frac{\partial}{\partial X_{nn}}V(X)\\
\end{array}
\right) ,
\end{equation}
which obviously is an $n\times n$ matrix.
The second-order gradient, which is named as Hessian operator and denoted by   $\nabla^2 V$, is defined by
\begin{equation}\label{Ren1}
\begin{split}
\nabla^2 V(X)
&=\left(\begin{array}{cccc}
\nabla\Big(\frac{\partial}{\partial X_{11}}V(X)\Big)&  \nabla\Big(\frac{\partial}{\partial X_{12}}V(X)\Big)&\cdots & \nabla\Big(\frac{\partial}{\partial X_{1n}}V(X)\Big)\\
\nabla\Big(\frac{\partial}{\partial X_{21}}V(X)\Big)&  \nabla\Big(\frac{\partial}{\partial X_{22}}V(X)\Big)&\cdots & \nabla\Big(\frac{\partial}{\partial X_{2n}}VX)\Big)\\
\cdots&   \cdots &\cdots&\cdots\\
\nabla\Big(\frac{\partial}{\partial X_{n1}}V(X)\Big)&\nabla\Big(\frac{\partial}{\partial X_{n2}}V(X)\Big)&
\cdots & \nabla\Big(\frac{\partial}{\partial
	X_{nn}}V(X)\Big)
\end{array}
\right),
\end{split}
\end{equation}
where each entry \Big(i.e.,  $\nabla\Big(\frac{\partial}{\partial X_{ij}}V(X)\Big)$\Big) is an $n\times n$ matrix and is defined as in \eqref{Ren}. Moreover, we remark 
that  $\nabla^2 V(X)$ is an $n^2\times n^2$ matrix which could be treated as an $n\times n$ block matrix in which each entry is an $n\times n$ square matrix. So we 
would like to introduce block matrix multiplication for our later calculations. 
 
For $A^{ij}\in\R^n\otimes\R^n$ and $B_{ij}\in\R$ with $1\le i,j\le n$, let 
\begin{equation*}
A=\left(\begin{array}{cccc}
A^{11}&   A^{12}&\cdots & A^{1n}\\
A^{21}&   A^{22}&\cdots & A^{2n}\\
\cdots& \cdots &\cdots &  \cdots\\
A^{n1}&   A^{n2}&\cdots & A^{nn}\\
\end{array}
\right) ~~~\mbox{ and }~~~B=\left(\begin{array}{cccc}
B_{11}&   B_{12}&\cdots & B_{1n}\\
B_{21}&   B_{22}&\cdots & B_{2n}\\
\cdots& \cdots &\cdots &  \cdots\\
B_{n1}&   B_{n2}&\cdots & B_{nn}\\
\end{array}
\right) .
\end{equation*} 
We then define the Hadamard product of $A$ and $B$ as follows  

\begin{equation*}
A\circ B:=\left(\begin{array}{cccc}
A^{11}B_{11}&   A^{12}B_{12}&\cdots & A^{1n}B_{1n}\\
A^{21}B_{21}&   A^{22}B_{22}&\cdots & A^{2n}B_{2n}\\
\cdots& \cdots &\cdots &  \cdots\\
A^{n1}B_{n1}&   A^{n2}B_{n2}&\cdots & A^{nn}B_{nn}\\
\end{array}
\right) .
\end{equation*} 
and further we define  

\begin{equation*}
\begin{split}
A\bullet  B=trace(A\circ B)&=A^{11}B_{11}+A^{12}B_{21}+\cdots+A^{1n}B_{n1}\\
&\quad+\cdots\\
&\quad+A^{n1}B_{1n}+A^{n2}B_{2n}+\cdots+A^{nn}B_{nn}\\
&=\sum_{ij=1}^nA^{ij}B_{ij}.
\end{split}
\end{equation*}

The following lemma is concerned with the Taylor expansion of functions with matrix-valued arguments, which plays a crucial role in proving the It\^o formula below. 
\begin{lemma}\label{Taylor}
	Let $V\in C^2(\R^n\otimes\R^n;\R)$. Then, for any $X=(X_{ij})_{n\times n}\in\R^n\otimes\R^n$ and $Y=(Y_{ij})_{n\times n}\in\R^n\otimes\R^n$,
	\begin{equation}\label{b1}
	\begin{split}
	V(X)&=V(Y)+\<\nabla V(Y),X-Y\>_{HS}\\
	&\quad+\frac{1}{2}\<(\nabla^2 V(Y+s(X-Y)))^T\bullet(X-Y),X-Y\>_{HS}
	\end{split}
	\end{equation}
	for some $s\in(0,1)$.
\end{lemma}

\begin{proof}
For the sake of brevity, 
we here just show \eqref{b1} holds true for the case  $n=2$ since the case $n>2$ can be done similarly. In the sequel, we assume $V\in C^2(\R^n\otimes\R^n;\R)$. For any   $X=(X_{ij})_{2\times 2}\in\R^2\otimes\R^2$ and $Y=(Y_{ij})_{2\times 2}\in\R^2\otimes\R^2$, set 
	\begin{equation*}
	\phi(s):=V(Y+s(X-Y)),~~~~s\in[0,1].
	\end{equation*}
	Note that $\phi(1)=V(X)$ and $\phi(0)=V(Y)$. By the Taylor expansion, one has
	\begin{equation}\label{b2}
	\phi(1)=\phi(0)+\phi'(0)+\frac{1}{2}\phi''(\xi)
	\end{equation}
	for some $\xi\in(0,1)$. In what follows, we are going to calculate $\phi'(s)$ and $\phi''(s)$ for $s\in[0,1].$  For $Z_s:=Y+s(X-Y), s\in[0,1]$, a straightforward calculation shows that 
	\begin{equation}\label{b3}
	\begin{split}
	\phi'(s)&=(X_{11}-Y_{11})\frac{\partial}{\partial X_{11}}V(Z_s)+(X_{12}-Y_{12})\frac{\partial}{\partial X_{12}}V(Z_s)\\
	&\quad+(X_{21}-Y_{21})\frac{\partial}{\partial X_{21}}V(Z_s)+(X_{22}-Y_{22})\frac{\partial}{\partial X_{22}}V(Z_s)\\
\end{split}
\end{equation}
From \eqref{Huaxianzi0} and \eqref{Ren} , we further have 
	\begin{equation}\label{bb}
	\begin{split}
\phi'(s)	&=\mbox{trace}\left(\left(\begin{array}{cc}
	\frac{\partial}{\partial X_{11}}V(Z_s)& \frac{\partial}{\partial X_{21}}V(Z_s)\\
	\frac{\partial}{\partial X_{12}}V(Z_s)&
	\frac{\partial}{\partial
		X_{22}}V(Z_s)
	\end{array}
	\right)\left(\begin{array}{cc}
	X_{11}-Y_{11} & X_{12}-Y_{12}\\
	X_{21}-Y_{21} & X_{22}-Y_{22}\\
	\end{array}
	\right)\right)\\
	&=\mbox{trace}\Big((\nabla V(Z_s))^T(X-Y)\Big)\\
	&=\<\nabla V(Z_s),X-Y\>_{HS}.
	\end{split}
	\end{equation}
	Next, following the procedure to derive \eqref{b3}, we deduce that 
	\begin{equation}\label{b5}
	\begin{split}
	\phi''(s)&=(X_{11}-Y_{11})A_{11}+(X_{12}-Y_{12})A_{12}+(X_{21}-Y_{21})A_{21}\\
	&\quad+(X_{22}-Y_{22})A_{22}\\
	&=\mbox{trace}\left(\left(\begin{array}{cc}
	A_{11}& A_{21}\\
	A_{12}& A_{22}\\
	\end{array}
	\right)\left(\begin{array}{cc}
	X_{11}-Y_{11} & X_{12}-Y_{12}\\
	X_{21}-Y_{21} & X_{22}-Y_{22}\\
	\end{array}
	\right)\right),\\
	\end{split}
	\end{equation}
	where
	\begin{align*}
	A_{11}:&=(X_{11}-Y_{11})\frac{\partial^2}{\partial X_{11}^2}V(Z_s)+(X_{12}-Y_{12})\frac{\partial^2}{\partial X_{11}\partial X_{12}}V(Z_s)\\
	&\quad+(X_{21}-Y_{21})\frac{\partial^2}{\partial X_{11}\partial X_{21}}V(Z_s)+(X_{22}-Y_{22})\frac{\partial^2}{\partial X_{11}\partial X_{22}}V(Z_s)\\
	A_{12}:&=(X_{11}-Y_{11})\frac{\partial^2}{\partial X_{11}\partial X_{12}}V(Z_s)+(X_{12}-Y_{12})\frac{\partial^2}{\partial X_{12}^2}V(Z_s)\\
	&\quad+(X_{21}-Y_{21})\frac{\partial^2}{\partial X_{12}\partial X_{21}}V(Z_s)+(X_{22}-Y_{22})\frac{\partial^2}{\partial X_{12}\partial X_{22}}V(Z_s)\\
	A_{21}:&=(X_{11}-Y_{11})\frac{\partial^2}{\partial X_{11}\partial X_{21}}V(Z_s)+(X_{12}-Y_{12})\frac{\partial^2}{\partial X_{12}\partial X_{21}}V(Z_s)\\
	&\quad+(X_{21}-Y_{21})\frac{\partial^2}{\partial X_{21}^2}V(Z_s)+(X_{22}-Y_{22})\frac{\partial^2}{\partial X_{22}\partial X_{21}}V(Z_s)\\
	A_{22}:&=(X_{11}-Y_{11})\frac{\partial^2}{\partial X_{11}\partial X_{22}}V(Z_s)+(X_{12}-Y_{12})\frac{\partial^2}{\partial X_{12}\partial X_{22}}V(Z_s)\\
	&\quad+(X_{21}-Y_{21})\frac{\partial^2}{\partial X_{21}\partial X_{22}}V(Z_s)+(X_{22}-Y_{22})\frac{\partial^2}{\partial X_{22}^2}V(Z_s).
	\end{align*}
Observe that 
		\begin{align*}
	\left(\begin{array}{cc}
	A_{11}& A_{21}\\
	A_{12}& A_{22}\\
	\end{array}
	\right)&=	(X_{11}-Y_{11})\left(\begin{array}{cc}
\frac{\partial^2}{\partial X_{11}^2}V(Z_s)& \frac{\partial^2}{\partial X_{11}\partial X_{21}}V(Z_s)\\
	\frac{\partial^2}{\partial X_{11}\partial X_{12}}V(Z_s)& \frac{\partial^2}{\partial X_{11}\partial X_{22}}V(Z_s)\\
	\end{array}
	\right)\\
	&\quad+	(X_{12}-Y_{12})\left(\begin{array}{cc}
\frac{\partial^2}{\partial X_{11}\partial X_{12}}V(Z_s)& \frac{\partial^2}{\partial X_{12}\partial X_{21}}V(Z_s)\\
\frac{\partial^2}{\partial X_{12}^2}V(Z_s)& \frac{\partial^2}{\partial X_{12}\partial X_{22}}V(Z_s)\\
	\end{array}
	\right)\\
	&\quad+(X_{21}-Y_{21})\left(\begin{array}{cc}
\frac{\partial^2}{\partial X_{11}\partial X_{21}}V(Z_s)&\frac{\partial^2}{\partial X_{21}^2}V(Z_s)\\
\frac{\partial^2}{\partial X_{12}\partial X_{21}}V(Z_s)& \frac{\partial^2}{\partial X_{21}\partial X_{22}}V(Z_s)\\
	\end{array}
	\right)\\
	&\quad+		(X_{22}-Y_{22})\left(\begin{array}{cc}
\frac{\partial^2}{\partial X_{11}\partial X_{22}}V(Z_s)& \frac{\partial^2}{\partial X_{22}\partial X_{21}}V(Z_s)\\
\frac{\partial^2}{\partial X_{12}\partial X_{22}}V(Z_s)&\frac{\partial^2}{\partial X_{22}^2}V(Z_s)\\
	\end{array}
	\right)\\
	&=\Big(\nabla\Big(\frac{\partial}{\partial X_{11}}\Big)\Big)^T(X_{11}-Y_{11})+\Big(\nabla\Big(\frac{\partial}{\partial X_{12}}\Big)\Big)^T(X_{12}-Y_{12})\\
	&\quad+\Big(\nabla\Big(\frac{\partial}{\partial X_{21}}\Big)\Big)^T(X_{21}-Y_{21})+\Big(\nabla\Big(\frac{\partial}{\partial X_{22}}\Big)\Big)^T(X_{22}-Y_{22})\\
	&=\left(\begin{array}{cc}
	\nabla\Big(\frac{\partial}{\partial X_{11}}V(Z_s)\Big)&   \nabla\Big(\frac{\partial}{\partial X_{12}}V(Z_s)\Big)\\
	\nabla\Big(\frac{\partial}{\partial X_{21}}V(Z_s)\Big)&   \nabla\Big(\frac{\partial}{\partial X_{22}}V()Z_s)\Big)\\
	\end{array}
	\right)^T\bullet\left(\begin{array}{cc}
	X_{11}-Y_{11} & X_{12}-Y_{12}\\
	X_{21}-Y_{21} & X_{22}-Y_{22}\\
	\end{array}
	\right)\\
	&=\Big(\nabla^2V(Z_s)\Big)^T\bullet(X-Y).
		\end{align*}
By substituting this into \eqref{b5} and combining with\eqref{b2} and \eqref{b3}, thus \eqref{b1} follows immediately. 
\end{proof}

We are now in the position to show It\^o formula the solutions of the SDE \eqref{Belfast1}, i.e. 
$$dX(t)=b(t,X(t))dt+\sigma(t,X(t))dB_t ,~~~~t>0,~~~~~X_0=x\in
\mathbb{R}^{n}\otimes \mathbb{R}^{n}. $$ 
We have the following 

\begin{theorem} \label{Itoformula}
	Assume that $V\in C^{1,2}(\mathbb{R}_{+}\times\mathbb{R}^{n}\otimes \mathbb{R}^{n};\R)$ and $\frac{\partial}{\partial t}V$, $\nabla V$
	and $\nabla^2V$ are continuous on bounded sets of $\R^n\otimes\R^n$. Suppose further that 
	 $(X(t))_{t\ge0}$ be the strong solution of \eqref{Belfast1}. Then,
	\begin{equation}\label{cup}
	\begin{split}
	dV(t,X(t))
	&=\frac{\partial}{\partial t}V(t, X(t)) d t+\<\nabla V(t,X(t)),b(t,X(t))\>_{HS}dt\\
	&\quad+\<\nabla V(t,X(t)),\sigma(t,X(t))dB_t\>_{HS}\\
	&\quad+\frac{1}{2}\mbox{trace}\bigg(\sigma^T(t,X(t))\Big((\nabla^2V(t,X(t)))^T\bullet\sigma(t,X(t))\Big)\bigg).
	\end{split}
	\end{equation}
\end{theorem}

\begin{proof}
We first assume that the stochastic process $(X(t))_{t\in[0,T]}$, $\int_0^T\|b(t,X(t))\|_{HS}d t$ and $\int_0^T\|\sigma(t,X(t))\|_{HS}^2d t$ are bounded.  To begin, we  show \eqref{cup} for the following $\mathbb{R}^{n}\otimes
	\mathbb{R}^{n}$-valued SDE
	\begin{equation}\label{cup1}
	dX(t)=b_0dt+\sigma_0dB_t ,~~~~t>0,~~~~~X_0=x\in\mathbb{R}^{n}\otimes
	\mathbb{R}^{n},
	\end{equation}
	where $b_0,\sigma_0\in \mathbb{R}^{n}\otimes \mathbb{R}^{n}$. Let
	$0=t_0< t_1<\cdots < t_{k-1}<t_k=t$ be a partition of $[0,t]$. Then
	we get that
	\begin{equation}\label{cup2}
	\begin{split}
	V(t,X(t))-V(0,X(0))
	&=\sum_{j=0}^{k-1}\Big(V(t_{j+1},X(t_{j+1}))-V(t_j,X(t_j))\Big)\\
	&=\sum_{j=0}^{k-1}\Big(V(t_{j+1},X(t_{j+1}))-V(t_j,X(t_{j+1}))\Big)\\
	&\quad-\sum_{j=0}^{k-1}\Big(V(t_j,X(t_{j+1}))-V(t_j,X(t_j))\Big)\\
	&=:I_1+I_2.
	\end{split}
	\end{equation}
	
	By the Taylor expansion, we deduce that
	\begin{equation}\label{cup3}
	\begin{split}
	I_1
	&=\sum_{j=0}^{k-1}\frac{\partial}{\partial t}V( \tilde{t}_j,X(t_{j+1}))\Delta t_j\\
	&=\sum_{j=0}^{k-1}\frac{\partial}{\partial t}V(t_{j+1},X(t_{j+1}))\Delta t_j\\
	&\quad+\sum_{j=0}^{k-1}\Big\{\frac{\partial}{\partial t}V(\tilde{t}_j,X(t_{j+1}))-\frac{\partial}{\partial t}V(t_{j+1},X(t_{j+1}))\Big\}\Delta t_j\\
	&=:J_1+J_2,
	\end{split}
	\end{equation}
	where $\tilde{t}_j:=t_j+\theta_{0j}\Delta t_j $ with $\Delta t_j:=t_{j+1}-t_j $ for some random variable $\theta_{0j}\in (0,1), j=0, \cdots, k-1.$ Let $\Delta t=\max_{j=0,1,\cdots,k-1}\{\Delta t_j\}$. By the
	definition of Riemann integral, note that $\mathbb{P}-a.s$
	$$J_1\rightarrow \int_{0}^{t}\frac{\partial}{\partial s}V(s,X(s))ds$$
	as $\Delta t\rightarrow 0$. Next, due to the uniform continuity of $\frac{\partial}{\partial t}V$ w.r.t the first variable, we have
	$$J_2\rightarrow 0~~~~~~~~~~~\mathbb{P}-a.s$$
	as $\Delta t\rightarrow0$. 
	Next, applying Lemma \ref{Taylor}
	to deduce that
	\begin{equation*} 
	\begin{split}
	I_2
	&=\sum_{j=0}^{k-1}\Big\<\nabla V(t_j,X(t_j)),\Delta X(t_j)\Big\>_{HS}\\
	&\quad+\frac{1}{2}\sum_{j=0}^{k-1}\Big\<\Big(\nabla^2V(t_j,\tilde{X}(t_j))\Big)^T\bullet \Delta X(t_j),\Delta X(t_j)\Big\>_{HS}\\
	&=\sum_{j=0}^{k-1}\Big\<\nabla V(t_j,X(t_j)),\Delta X(t_j)\Big\>_{HS}\\
	&\quad+\frac{1}{2}\sum_{j=0}^{k-1}\Big\<\Big(\nabla^2 V(t_j,X(t_j))\Big)^T\bullet \Delta X(t_j),\Delta X(t_j)\Big\>_{HS}\\
	&\quad+\frac{1}{2}\sum_{j=0}^{k-1}\Big\<\Big(\nabla^2V(t_j,\tilde{X}(t_j)-\nabla^2 V(t_j,X(t_j))\Big)^T\bullet \Delta X(t_j),\Delta X(t_j)\Big\>_{HS}\\
	&=:\Lambda_1+\Lambda_2+\Lambda_3.
	\end{split}
	\end{equation*}
	From \eqref{cup1}, it follow that
	\begin{equation}\label{cup7}
	\Delta X(t_j)=b_0\Delta t_j+\sigma_0 \Delta B(t_j)
	\end{equation}
	where $\Delta B(t_j):=B(t_{j+1})-B(t_j)$. Now, taking
	\eqref{cup7} into consideration, we infer that
	\begin{equation*}
	\begin{split}
\Lambda_1
	&=\sum_{j=0}^{k-1}\<\nabla V(t_j,X(t_j)), b_0\>_{HS}\Delta
	t_j\\
	&\quad+\sum_{j=0}^{k-1}\<\nabla V(t_j,X(t_j)), \sigma_0 \Delta
	B(t_j)\>_{HS}
	\end{split}
	\end{equation*}
	and that
	\begin{align*}\label{cup9}
	\Lambda_2
	&=\frac{1}{2}\sum_{j=0}^{k-1}\Big\<\Big(\nabla^2 V(t_j, X(t_j))\Big)^T\bullet b_0, b_0\Big\>_{HS}(\Delta t_j)^2\\
	&\quad+\frac{1}{2}\sum_{j=0}^{k-1}\Big\<\Big(\nabla^2 V(t_j, X(t_j))\Big)^T\bullet b_0,\sigma_0 \Delta B(t_j)\Big\>_{HS}\Delta t_j\\
	&\quad+\frac{1}{2}\sum_{j=0}^{k-1}\Big\<\Big(\nabla^2 V(t_j, X(t_j))\Big)^T\bullet \sigma_0 \Delta B(t_j), b_0\Big\>_{HS}\Delta t_j\\
	&\quad+\frac{1}{2}\sum_{j=0}^{k-1}\Big\<\Big(\nabla^2 V(t_j, X(t_j))\Big)^T\bullet \sigma_0 \Delta B(t_j), \sigma_0 \Delta B(t_j)\Big\>_{HS}\\
	&=:\Lambda_{11}+\Lambda_{12}+\Lambda_{13}+\Lambda_{14}.
	\end{align*}
	As $\Delta t\rightarrow0 $, we obtain that
	$\Lambda_{11}+\Lambda_{12}+\Lambda_{13}\rightarrow0,~\mathbb{P}-a.s$
	and that
	$$
	\Lambda_2\rightarrow	\Lambda_{14}\rightarrow
	\frac{1}{2}\int_{0}^{t}\mbox{trace}\bigg(\sigma_0^T\Big(\Big(\nabla
	^2V(s,X(s)\Big)^T\bullet\sigma_0)\Big)\bigg)ds~~~~\mathbb{P}-\mbox{a.s.}
	$$
	Moreover, 
 we have
	$$\Lambda_2\rightarrow 0~~~~\mathbb{P}-\mbox{a.s.}$$
	So, we conclude that \eqref{cup} holds for equation\eqref{cup1}.
	Also It\^{o}'s formula holds when $b$ and $\sigma$ are simple
	stochastic process. If $b,\sigma\in\mu^2_T(\mathbb{R}^n\otimes \mathbb{R}^{n})$,
	according to Lemma \ref{app},
	there exist   sequences of simple  processes $b_k$ and
	$\sigma_k$  such that 
	$$
	\int_{0}^{t}\E\lVert b(s,X(s))-b_k(s)\rVert^2_{HS}ds+
	\int_{0}^{t}\E\lVert \sigma(s,X(s))-\sigma_k(s)\rVert^2_{HS}ds\rightarrow0
	$$
	as $k\rightarrow \infty.$ Then, 
	\eqref{cup} still holds by a standard approximation argument. 
	For the general case, we adopt the standard stopping time approach.

\end{proof}

As an application of the It\^o formula, let us present anther set of sufficient conditions for the existence and 
uniqueness of the SDE \eqref{Belfast1}.  More precisely, we replace  {\bf (A1)} and {\bf (A2)} by the following local Lipschitz condition and the monotone condition.
\begin{enumerate}
	\item[({\bf H1})] For all $t\in[0,T],R>0$ and $x,y\in\mathbb{R}^{n}\otimes \mathbb{R}^{n}$ with $\lVert x\lVert_{HS}\vee\lVert y\lVert_{HS}\leq R
	$, there exists a non-decreasing integrable  function $\kk^{(R)}_\cdot:[0,T]\rightarrow\R_+$ such that
	$$
	\lVert b(t,x)-b(t,y)\rVert^2_{HS}+\lVert\sigma(t,x)-\sigma(t,y)\rVert^2_{HS} \leq \kk_t^{(R)}\lVert
	x-y\rVert^2_{HS}.
	$$
	
	\item[({\bf H2})]  For any $t\in[0,T]$ and $x\in\mathbb{R}^{n}\otimes
	\mathbb{R}^{n}$, there exists an integrable function $\kk:[0,T]\rightarrow\R_+$
	$$
	2\langle x, b(t,x)\rangle_{HS}+\lVert\sigma(t,x)\rVert^2_{HS}\leq \kk_t(1+\lVert x\rVert^2_{HS}).
	$$
\end{enumerate}
From \eqref{66}, in addition to Schwartz's inequality and the inequality: $2ab\le a^2+b^2, a,b\in\R,$ we find that
\begin{equation*}
\begin{split}
&2\langle x, b(t,x)\rangle_{HS}+\lVert\sigma(t,x)\rVert^2_{HS}\\
&\le \|x\|_{HS}^2+\|b(t,x)\|_{HS}^2+\|\sigma(t,x)\|_{HS}^2\\
&\le  (1+2\kk(t))(1+\|x\|_{HS}^2).
\end{split}
\end{equation*}
So, {\bf (A1)} and {\bf (A2)} implies the monotone condition  {\bf(H2)}.

\begin{theorem}\label{pull}
	Under ${(\bf{A2})}$, ${(\bf{H1})} $and ${(\bf{H2})}$, \eqref{Belfast1} has a unique global strong solution $(X(t))_{t\geq0}$. Furthermore, if there exists an integrable function $\kk_0:[0,T]\rightarrow\R_+$ such that
	\begin{equation}\label{666}
	\lVert\sigma(t,x)\rVert^2_{HS}\le\kappa_0(t)(1+\|x\|_{HS}^2),
	\end{equation}
	then
	\begin{equation}\label{777}
	\E\Big( \sup_{0\leq t\leq T}\lVert X(t)\rVert^2_{HS} \Big)\leq(
	1+2\lVert x\rVert^2_{HS}) e^{2\int_0^T(\kk_t+64\kk_0(t))d t}.
	\end{equation}
	
\end{theorem}

\begin{proof}
	We take advantage of  the standard truncation approach to show existence of a solution to \eqref{Belfast1}, see, e.g., \cite{AlbBrzWu}. For any $R>0$, $t\in[0,T],$ and   $x\in \mathbb{R}^{n}\otimes \mathbb{R}^{n}$, let's define
	\begin{equation}
	b_R(t,x):=b\Big(t, \frac{\lVert x \rVert_{HS}\wedge R}{\lVert x
		\rVert_{HS}} x\Big),~~~\sigma_R(t,x):=\sigma\Big(t, \frac{\lVert
		x\rVert_{HS}\wedge R}{\lVert x\rVert_{HS}} x\Big).
	\end{equation}
	In lieu of  \eqref{Belfast1}, we consider the following $\R^n\otimes\R^n$-valued SDE
	\begin{equation}\label{pull5}
	dX^R(t)=b_R(t,X^R(t))dt+\sigma_R(t,X^R(t))dB(t),~~~t>0
	\end{equation}
	with the initial value $X^R(0)=x\in \mathbb{R}^{n}\otimes \mathbb{R}^{n}.$ A straightforward calculation shows that, for any $x,y\in\R^n\otimes\R^n,$
	\begin{equation*}
	\begin{split}
	&\lVert b_R(t,x)-b_R(t,y)\rVert^2_{HS}+\lVert\sigma_R(t,x)-\sigma_R(t,y)\rVert^2_{HS}\\
		&\leq k_t(R)\Big\rVert \frac{\lVert x \rVert_{HS}\wedge R}{\lVert x \rVert_{HS}} x -\frac{\lVert y \rVert_{HS}\wedge R}{\lVert y \rVert_{HS}} y\Big\rVert_{HS}^2\\
		&=\begin{cases}
		k_t(R)\Big\lVert x-y \Big\rVert_{HS}^2     \mbox{~~~~~~~~~~~~~~~~~~~~~for~~} \lVert x\rVert_{HS}\vee \lVert y\rVert_{HS}\leq R;\\
		k_t(R)\Big\lVert x- \frac{R}{\lVert y\rVert_{HS}}y  \Big\rVert_{HS}^2 \mbox{~~~~~~~~~~~~~for~~} \lVert x\rVert_{HS}\leq R, \lVert y\rVert_{HS}\geq R;\\
		k_t(R)\Big\lVert \frac{R}{\lVert x\rVert_{HS}}x-y\Big\rVert_{HS}^2    \mbox{~~~~~~~~~~~~~for~~} \lVert x\rVert_{HS}\geq R, \lVert y\rVert_{HS}\leq R;\\
		k_t(R)\Big\lVert \frac{R}{\lVert x\rVert_{HS}}x-\frac{R}{\lVert y\rVert_{HS}}y\Big\rVert_{HS}^2  \mbox{~~~~~~for~~} \lVert x\rVert_{HS}\geq R, \lVert y\rVert_{HS}\geq R;\\
		\end{cases}\\
		&\leq 4k_t^{(R)}\lVert x-y\rVert_{HS}^2.
		\end{split}
	\end{equation*}
	Namely, for fixed $R>0,$ $b_R$ and $\sigma_R$ satisfy the global Lipschitz condition {\bf (A1)} with $\kk_1(t)$ being replaced by $4k_t^{(R)}$. Hence,   besides {\bf (A2)}, we conclude that \eqref{pull5} has a unique strong solution $(X^R(t))_{t\ge0}$ according to Theorem \ref{push}.
	
	For any $R>\|x\|_{HS}$, define the stopping time  as
	\begin{equation*}
	\tau_{R}=\inf\{t\geq 0:\lVert X^R(t) \lVert_{HS}\geq R\}.
	\end{equation*}
	Also, notice that
	\begin{equation}\label{open1}
	dX^{R+1}(t)=b_R(t,X^{R+1}(t))dt+\sigma_R(t,X^{R+1}(t))dB(t),  \mbox{~~~~~}t\in [0,\tau _R]
	\end{equation}
with the initial value $X^{R+1}(0)=x\in\R^n\otimes\R^n$.
	In terms of the uniqueness of solution to \eqref{pull5}, one has
	$$X^{R}(t)=X^{R+1}(t)=X^{R+2}(t)=\cdots, \mbox{~~~~~}t\in [0,\tau _R],$$
	which further implies that
	$$\tau _R\leq \tau _{R+1}\leq \cdots.$$
	That is, $\tau_R$ is non-decreasing w.r.t. $R$. So, we can define $\tau _{\infty}=\lim_{R\rightarrow \infty}\tau_{R}. $ Furthermore, let's define $ X(t)=X^{R}(t),t\in[0,\tau_R]. $ Then, $(X(t))_{0\le t<\tau_\infty}$ is a  local maximal solution to \eqref{Belfast1}. Next, we show that the local maximal solution is in fact a global solution, i.e., $\tau_\infty=\infty$ or $\rho_k=\infty$, a.s., where, for $k>\|x\|_{HS}$,
	\begin{equation*}
	\rho_k:=\inf\{t\in(0,\tau_\infty):\|X(t)\|_{HS}\ge k\}.
	\end{equation*} 
	By It\^o's formula \eqref{cup}, it follows from {\bf(H2)} that
	\begin{equation*}\label{open2}
	\begin{split}
	&\E\lVert X(t\wedge \rho_k) \lVert_{HS}^2\\
	&=\lVert x\lVert_{HS}^2+\E\int _{0}^{t\wedge\rho_k }\{2\langle X(s),b(s,X(s))\rangle+\|\sigma(s,X(s))\|_{HS}^2\}ds\\
	&\leq\lVert x\lVert_{HS}^2+\E\int _{0}^{t\wedge\rho_k
	}\kk_s(1+\|X(s)\|_{HS}^2)ds\\
	&=\lVert x\lVert_{HS}^2+\E\int _{0}^{t\wedge\rho_k
	}\kk_{s\wedge\rho_k
}(1+\|X(s\wedge\rho_k)\|_{HS}^2)ds\\
&\leq\lVert x\lVert_{HS}^2+\int _{0}^{t
}\kk_s(1+\E\|X(s\wedge\rho_k)\|_{HS}^2)ds,
\end{split}
\end{equation*}
where in the last display we have used the non-decreasing property of $\kk_\cdot.$ Then,  the Gronwall inequality leads to
\begin{equation*}\label{open3}
\E\lVert X(t\wedge \rho_k) \lVert_{HS}^2\leq C_t:=\Big(\lVert
x\lVert_{HS}^2+\int _{0}^{t}\kk_{s}ds\Big) e^{\int _{0}^{t}\kk_{s}ds}.
\end{equation*}
This, combining with the definition of the stopping time $\rho_k,$ further implies that
\begin{equation*}\label{open4}
\begin{split}
C_t&\geq \E(\lVert X(t\wedge \rho_k) \lVert_{HS}^2I_{\{t\geq \rho_k\}})+ \E(\lVert X(t\wedge \rho_k) \lVert_{HS}^2I_{\{t< \rho_k\}})\\
&\geq \E(\lVert X(t\wedge \rho_k) \lVert_{HS}^2I_{\{t\geq \rho_k\}})\\
&=\E(\lVert X(\rho_k) \lVert_{HS}^2I_{\{t\geq \rho_k\}})\\
&=k^2\mathbb{P}(t\geq \rho_k),
\end{split}
\end{equation*}
where $I_{\{\cdot\}} $ denotes the indicator function associated with the set $\{\cdot\}$. Thus, we obtain from Chebyshev's  inequality that
\begin{equation}\label{open5}
\mathbb{P}(t\geq \rho_k)\leq C_t/k^2.
\end{equation}
Since  $\sum_{k=1}^{\infty}\mathbb{P}(t\geq \rho_k)\le
C_t\sum_{k=1}^{\infty} \frac{1}{k^2}$, which is  convergent,  the Borel-Cantelli lemma implies that
$$ \mathbb{P}(t\geq \tau_{\infty})=0.$$
So $\tau_\infty=\infty$, a.s. due to the arbitrariness of $t$. Then we complete the existence and uniqueness for \eqref{Belfast1} under the conditions {(\bf{A2})}, {(\bf{H1})} and {(\bf{H2})}. Now, we show the second part of this theorem.

Also, by It\^o's formula \eqref{cup}, it follows that
\begin{equation*}
d X(t)=\{2\langle
X(t),b(t,X(t))\rangle_{HS}+\|\sigma(t,X(t))\|_{HS}^2\}d t+2\langle
X(t),\sigma(t,X(t))d B(t)\rangle_{HS}.
\end{equation*}
This, in addition to {\bf(H2)}, yields that
\begin{equation}\label{88}
\begin{split}
&\E\Big(\sup_{0\le s\le t}\|X(s)\|_{HS}^2\Big)\\
&\le\|x\|_{HS}^2+\Upsilon(t)+\E\Big(\sup_{0\le s\le
	t}\Big\|\int_0^s\{2\langle
X(u),b(u,X(u))\rangle_{HS}\\
&\quad+\|\sigma(u,X(u))\|_{HS}^2\}d u\Big)\\
&\le\|x\|_{HS}^2+\Upsilon(t)+\int_0^t\kk_s(1+\E\lVert
X(s)\rVert^2_{HS}) d s,
\end{split}
\end{equation}
where
\begin{equation*}
\Upsilon(t):=2\E\Big(\sup_{0\le s\le t}\Big\|\int_0^s\langle
X(u),\sigma(u,X(u))d B(u)\rangle_{HS}\Big\|_{HS}^2\Big).
\end{equation*}
For any $u\in[0,T]$, define
\begin{equation*}
\psi(u)=
\begin{cases}
\frac{\sigma^T(u,X(u))X(u)}{\|\sigma^T(u,X(u))X(u)\|_{HS}},~~~~\sigma^T(u,X(u))X(u)\neq0,\\
e_{11},~~~~~~~~~~~~~~~~~~~~~~~~~\sigma^T(u,X(u))X(u)=0.
\end{cases}
\end{equation*}
It is easy to see that $\|\psi(u)\|_{HS}=1$. Set 
\begin{equation*}
\beta(t)=\int_0^t\<\psi(s),d B_s\>. 
\end{equation*}
Then, $\beta(t)$ is a martingale, where  the quadratic variation is $t$ due to $\|\psi(u)\|_{HS}=1$. Then, according to the L\'{e}vy characterization of Brownian motion, we conclude that $\beta(t)$ is an $\R$-valued Brownian motion. Note that 
\begin{equation*}
\begin{split}
&\int_0^s\|\sigma^T(u,X(u))X(u)\|_{HS}d \beta(u)\\
&=\int_0^s\|\sigma^Tu,X(u))X(u)\|_{HS} \<\psi(u),d B_u\>_{HS}\\
&=\int_0^s \<\sigma^T(u,X(u))X(u),d B_u\>_{HS}.
\end{split}
\end{equation*}
Hence, one has 
\begin{equation*}
\Upsilon(t):=2\E\Big(\sup_{0\le s\le t}\Big|\int_0^s\|\sigma(u,X(u)))^TX(u)\|_{HS}d \beta(u)\Big|^2\Big).
\end{equation*}
As a result, by the Burkhold-Davis-Gundy inequality, $2ab\le\varepsilon a^2+b^2/\varepsilon$ for any $a,b\in\R$ and $\varepsilon>0$ as well as \eqref{666},
\begin{equation}\label{888}
\begin{split}
\Upsilon(t)&\le8\sqrt{2}\E\Big(\int_0^t\|\sigma^T(s,X(s))X(s)\|_{HS}^2d
s\Big)^{1/2}\\
&\le8\sqrt{2}\E\Big(\sup_{0\le s\le
	t}\|X(s)\|_{HS}^2\int_0^t\|\sigma(s,X(s))\|_{HS}^2d s\Big)^{1/2}\\
&\le\frac{1}{2}\E\Big(\sup_{0\le s\le
	t}\|X(s)\|_{HS}^2\Big)+64\int_0^t\E\|\sigma(s,X(s))\|_{HS}^2 d s\\
&\le\frac{1}{2}\E\Big(\sup_{0\le s\le
	t}\|X(s)\|_{HS}^2\Big)+64\int_0^t\kk_0(s)(1+\E\|X(s)\|_{HS}^2) d s.
\end{split}
\end{equation}
Inserting \eqref{888} into \eqref{88}, we derive that
\begin{equation*}
\begin{split}
\E\Big(\sup_{0\le s\le t}\|X(s)\|_{HS}^2\Big)
&\le\|x\|_{HS}^2+\frac{1}{2}\E\Big(\sup_{0\le s\le
	t}\|X(s)\|_{HS}^2\Big)\\
&\quad+\int_0^t(\kk_s+64\kk_0(s))(1+\E\lVert
X(s)\rVert^2_{HS}) d s.
\end{split}
\end{equation*}
Hence,
\begin{equation*}
\begin{split}
\E\Big(\sup_{0\le s\le t}\|X(s)\|_{HS}^2\Big) \le2\|x\|_{HS}^2
+2\int_0^t(\kk_s+64\kk_0(s))(1+\E\lVert
X(s)\rVert^2_{HS}) d s.
\end{split}
\end{equation*}
Thus, by Gronwall's inequality we obtain the  inequality \eqref{777}.
\end{proof}

Comparing with Theorem \ref{push}, we remark that Theorem \ref{pull} provides much weaker conditions to guarantee the existence and uniqueness of \eqref{Belfast1}, more precisely, 
under local Lipschitz condition and monotone condition.

\end{document}